\documentclass[final,leqno,onefignum,onetabnum]{siamltex1213}
\usepackage{hyperref}
\usepackage{graphicx,epsfig,color}
\usepackage{bm}
\usepackage{booktabs}
\usepackage[notref,notcite]{showkeys}
\usepackage{amsmath,cases}

\usepackage{amssymb}

\renewcommand{\theequation}{\arabic{section}.\arabic{equation}}
\numberwithin{equation}{section}

\input{siam.sty}

\title{Preconditioning rectangular spectral collocation
} 

\author{Kui Du\thanks{
School of Mathematical Sciences and Fujian Provincial Key Laboratory of Mathematical Modeling and High-Performance Scientific Computation, Xiamen University, Xiamen 361005, China ({kuidu@xmu.edu.cn}). The research of this author was supported by the National Natural Science Foundation of China (No.11201392 and No.91430213), the Doctoral Fund of Ministry of Education of China (No.20120121120020), and the Fundamental Research Funds for the Central Universities (No.2013121003).} 
}

\begin{document}
\maketitle
\slugger{mms}{xxxx}{xx}{x}{x--x}

\begin{abstract}  
Rectangular spectral collocation (RSC) methods have recently been proposed to solve linear and nonlinear differential equations with general boundary conditions and/or other constraints. The involved linear systems in RSC become extremely ill-conditioned as the number of collocation points increase. By introducing suitable Birkhoff-type interpolation problems, we present pseudospectral integration preconditioning matrices for the ill-conditioned linear systems in RSC. The condition numbers of the preconditioned linear systems are independent of the number of collocation points. Numerical examples are given.
\end{abstract}



\pagestyle{myheadings}
\thispagestyle{plain}
\markboth{Kui Du}{Preconditioning rectangular spectral collocation}

\section{Introduction}  
Rectangular spectral collocation methods \cite{driscoll2015recta} have recently been demonstrated to be a convenient means of solving the problems when the row replacement or `boundary bordering' strategy of standard spectral collocation methods \cite{fornberg1996pract,trefethen2000spect,boyd2001cheby,canuto2006spect,shen2011spect} becomes ambiguous. Specifically, an $m$th-order differential operator is discretized by a rectangular matrix directly, allowing $m$ constraints to be appended to form an invertible square system.   
However, the involved linear systems become extremely ill-conditioned as the number of collocation points increases. Typically, the condition number grows like $N^{2m}$. Efficient preconditioners are highly required when solving the linear systems by an iterative method. 

Recently, Wang, Samson, and Zhao \cite{wang2014well} proposed a well-conditioned collocation method to solve linear differential equations with various types of boundary conditions. By introducing a suitable Birkhoff interpolation problem \cite{shi2003theor}, they constructed a pseudospectral integration preconditioning matrix, which is the exact inverse of the pseudospectral discretization matrix of the $m$th-order derivative operator together with $m$ boundary conditions. In this paper, we employ the similar idea to construct a pseudospectral integration matrix, which is the exact inverse of the discretization matrix arising in the rectangular spectral collocation method for $m$th-order derivative operator together with $m$ general linear constraints. The condition number of the resulting linear system is independent of the number of collocation points when the new pseudospectral integration matrix is used as a right preconditioner for an $m$th-order linear differential operator together with the same constraints.   

The rest of the paper is organized as follows. In \S 2, we review several topics required in the following sections. In \S 3, we introduce the new pseudospectral integration matrix by a suitable Birkhoff-type interpolation problem. In \S 4, we present the preconditioning rectangular spectral collocation method. Numerical examples are reported in \S 5. We present brief concluding remarks in \S 6.

\section{Preliminaries} 
\subsection{Barycentric resampling matrix} Let $\{x_j\}_{j=0}^N$ be a set of distinct interpolation points satisfying \beq\label{iptsx} -1\leq x_0<x_1<\cdots <x_{N-1}<x_N\leq 1.\eeq
The associated barycentric weights are defined by \beq\label{bweight}w_{j,N}=\prod_{n=0,n\neq j}^N(x_j-x_n)^{-1}, \qquad j=0,1,\ldots,N.\eeq 
Let $\{y_j\}_{j=0}^M$ be another set of distinct interpolation points satisfying \beq\label{iptsy}-1\leq y_0< y_1< \cdots <y_{M-1}<y_M\leq 1.\eeq The barycentric resampling matrix \cite{driscoll2015recta}, ${\bf P}^{{\bf x}\mapsto {\bf y}}\in{\mbbr^{(M+1)\times(N+1)}}$, which interpolates between the points $\{x_j\}_{j=0}^N$ and $\{y_j\}_{j=0}^M$, is defined by \beqs{\bf P}^{{\bf x}\mapsto {\bf y}}=[p^{{\bf x}\mapsto {\bf y}}_{ij}]_{i=0,j=0}^{M,N},\eeqs where \beqs p^{{\bf x}\mapsto {\bf y}}_{ij}=\l\{\begin{array}{ll} \dsp\frac{w_{j,N}}{y_i-x_j}\l(\sum_{l=0}^N\dsp\frac{w_{l,N}}{y_i-x_l}\r)^{-1}, & y_i\neq x_j, \\ 1, & y_i=x_j.
\end{array} \r.\eeqs 
\begin{lemma}\label{mnnm}If $N\geq M$, then ${\bf P}^{{\bf x}\mapsto {\bf y}}{\bf P}^{{\bf y}\mapsto {\bf x}}={\bf I}_{M+1}.$ \end{lemma}

\subsection{Pseudospectral differentiation matrices}\label{sdm} 

The Lagrange interpolation basis polynomials of degree $N$ associated with the points $\{x_j\}_{j=0}^N$ are defined by \beqs\label{cf} \ell_{j,N}(x)=w_{j,N}\prod_{n=0,n\neq j}^N(x-x_n),\qquad j=0,1,\ldots,N,\eeqs where $w_{j,N}$ is the barycentric weight (\ref{bweight}). Define the pseudospectral differentiation matrices: \beqs\label{psdm} {\bf D}^{(m)}_{{\bf x}\mapsto {\bf x}}=\l[\ell_{j,N}^{(m)}(x_i)\r]_{i,j=0}^N,\qquad {\bf D}^{(m)}_{{\bf x}\mapsto {\bf y}}=\l[\ell_{j,N}^{(m)}(y_i)\r]_{i=0,j=0}^{M,N}.
\eeqs There hold $${\bf D}^{(m)}_{{\bf x}\mapsto {\bf x}}=\l({\bf D}^{(1)}_{{\bf x}\mapsto {\bf x}}\r)^m,\qquad m\geq 1,$$ and $${\bf D}^{(m)}_{{\bf x}\mapsto {\bf y}}={\bf P}^{{\bf x}\mapsto {\bf y}}{\bf D}^{(m)}_{{\bf x}\mapsto {\bf x}}.$$

The matrix ${\bf D}^{(m)}_{{\bf x}\mapsto {\bf y}}$ is called a rectangular $m$th-order differentiation matrix, which maps values of a polynomial defined on $\{x_j\}_{j=0}^N$ to the values of its $m$th-order derivative on $\{y_j\}_{j=0}^M$. Explicit formulae and recurrences for rectangular differentiation matrices are given in \cite{xu2015expli}. 

\subsection{Chebyshev polynomials and Chebyshev points}

The most widely used spectral methods for non-periodic problems are those based on Chebyshev polynomials and Chebyshev points. In this paper, we focus on these polynomials and points. However, everything we discuss can be easily generalized to the case of Jacobi polynomials and corresponding points. 

The { Chebyshev points of the first kind} (also known as { Gauss-Chebyshev points}) are given by \beqs\label{chebpts1} \nu_{j,N}=-\cos\frac{(2j+1)\pi}{2N+2},\qquad j=0,1,\ldots,N.\eeqs
In this case, the Gauss-Chebyshev quadrature weights are given by {\rm\cite{funaro1992polyn}} \beqs\label{gcw}\omega_{j,N}^\nu=\frac{\pi}{N+1},\qquad j=0,1,\ldots,N,\eeqs and the barycentric weights are given by \cite{henrici1982essen} \beqs\label{bw1}w_{j,N}^\nu=(-1)^{N-j}\frac{2^N}{N+1}\sin\frac{(2j+1)\pi}{2N+2},\qquad j=0,1,\ldots,N.\eeqs Let $\mbbp_n$ be the set of all algebraic polynomials of degree at most $n$. We have \beq\label{exact1}\int_{-1}^1 \frac{p(x)}{\sqrt{1-x^2}}\rmd x=\sum_{j=0}^N\omega_{j,N}^\nu p(\nu_{j,N}), \qquad \forall p(x)\in\mbbp_{2N+1}.\eeq 

The Chebyshev points of the second kind (also known as Gauss-Chebyshev-Lobatto points) are given by \beqs\label{chebpts2} \tau_{j,N}=-\cos\frac{j\pi}{N},\qquad j=0,1,\ldots,N.\eeqs In this case, the Gauss-Chebyshev-Lobatto quadrature weights are given by {\rm\cite{funaro1992polyn}} \beqs\label{gclw}\omega_{j,N}^\tau=\frac{\pi}{\rho_jN},\qquad j=0,\ldots,N,\eeqs and the barycentric weights are given by {\rm\cite{salzer1972lagra}} \beqs\label{bw2}w_{j,N}^\tau=(-1)^{N-j}\frac{2^{N-1}}{\rho_jN},\qquad j=0,1,\ldots,N,\eeqs where $$\rho_0=\rho_N=2, \qquad \rho_1=\rho_2=\cdots=\rho_{N-1}=1.$$ We have \beqs\label{exact2}\int_{-1}^1 \frac{p(x)}{\sqrt{1-x^2}}\rmd x=\sum_{j=0}^N\omega_{j,N}^\tau p(\tau_{j,N}), \qquad \forall p(x)\in\mbbp_{2N-1}.\eeqs 

Let $T_n(x)$ be the Chebyshev polynomials (see, for example, \cite{funaro1992polyn}) given by 
$$T_n(x)=\cos(n\arccos x),\qquad x\in[-1,1].$$ They are mutually orthogonal: \beq\label{orth}\int_{-1}^1\frac{T_k(x)T_j(x)}{\sqrt{1-x^2}}\rmd x=\frac{\varrho_k\pi}{2}\delta_{kj},\eeq where $$\varrho_k=\l\{\begin{array}{ll} 2, & k=0, \\ 1, & k\geq 1,\end{array}\r.\qquad \delta_{kj}=\l\{\begin{array}{ll} 1, & k=j, \\ 0, & k\neq j.\end{array}\r.$$ 

Let $\{\ell_{j,M}(x)\}_{j=0}^M$ denote the Lagrange interpolation basis polynomials of degree $M$ associated with the points $\{y_j\}_{j=0}^M$. The polynomial $\ell_{j,M}(x)$ can be rewritten as \beq\label{ltp}\ell_{j,M}(x)=\sum_{k=0}^M\beta_{kj} T_k(x), \qquad j=0,1,\ldots, M.\eeq 
If $\{y_j\}_{j=0}^M$ is a subset of $\{\nu_{j,N}\}_{j=0}^N$ or $\{\tau_{j,N}\}_{j=0}^N$, $\beta_{kj}$ can be obtained with ease. For example, suppose that $\{y_j\}_{j=0}^M$ is a proper subset of $\{\nu_{j,N}\}_{j=0}^N$. Let $\varpi(\cdot)$ denote the map such that $y_{j}=\nu_{\varpi(j),N}$. Let $\mcali$ denote the set such that if $i\in\mcali$ then $\nu_{i,N}\notin\{y_j\}_{j=0}^M$. By (\ref{exact1}) and (\ref{orth}), we have, for $j=0,1,\ldots,M,$ $$\beta_{kj}=\frac{2}{\varrho_k\pi}\l(T_k(y_j)\omega_{\varpi(j),N}^\nu+\sum_{i\in\mcali}T_k(\nu_{i,N})\omega_{i,N}^\nu \ell_{j,M}(\nu_{i,N})\r),\qquad k=0,1,\ldots,M.$$ Here $\ell_{j,M}(\nu_{i,N})$, $i\in\mcali$, can be obtained by solving the following linear system $$T_k(y_j)\omega_{\varpi(j),N}^\nu+\sum_{i\in\mcali}T_k(\nu_{i,N})\omega_{i,N}^\nu \ell_{j,M}(\nu_{i,N})=0, \qquad k=M+1,\ldots, N.$$ 

In particular, if $\{\ell_{j,N}^\nu(x)\}_{j=0}^N$ denote the Lagrange interpolation basis polynomials of degree $N$ associated with $\{\nu_{j,N}\}_{j=0}^N$, we have \beqs\label{lt1}\ell_{j,N}^\nu(x)=\sum_{k=0}^N\beta_{kj}^\nu T_k(x), \qquad j=0,1,\ldots, N,\eeqs where, for $j=0,1,\ldots, N,$  \beqas &&\beta_{0j}^\nu=\frac{1}{N+1}, \\ && \beta_{kj}^\nu=\frac{2 T_k(\nu_{j,N})}{N+1},\qquad k=1,2,\ldots, N-1, \\ &&\beta_{Nj}^\nu=(-1)^{N-j}\frac{2}{N+1}\sin\frac{(2j+1)\pi}{2N+2}.\eeqas  
If $\{\ell_{j,N}^\tau(x)\}_{j=0}^N$ denote the Lagrange interpolation basis polynomials of degree $N$ associated with $\{\tau_{j,N}\}_{j=0}^N$, we have \beqs\label{lt2}\ell_{j,N}^\tau(x)=\sum_{k=0}^N\beta_{kj}^\tau T_k(x), \qquad j=0,1,\ldots, N,\eeqs where, for $j=0,1,\ldots, N,$  \beqas &&\beta_{0j}^\tau=\frac{1}{\rho_j N}, \\ && \beta_{kj}^\tau=\frac{2 T_k(\tau_{j,N})}{\rho_j N},\qquad k=1,2,\ldots, N-1, \\ &&\beta_{Nj}^\tau=\frac{(-1)^{N-j}}{\rho_j N}.\eeqas 

Define the integral operators: $$\p_x^{-1} v(x)=\int_{-1}^x v(t)\rmd t; \qquad \p_x^{-k}v(x)=\p_x^{-1}\l(\p_x^{-(k-1)}v(x)\r), \quad k\geq 2.$$ By \beqs\label{recursive}T_n(x)=\frac{T_{n+1}'(x)}{2(n+1)}-\frac{T_{n-1}'(x)}{2(n-1)},\qquad n\geq 2,\eeqs and \beqs\label{lrbp}T_n(\pm1)=(\pm1)^n, \qquad T_n'(\pm1)=\pm(\pm1)^nn^2,\eeqs we have
\beqa && \p_x^{-1}T_0(x)=1+x,\nn \\ && \label{px1}\p_x^{-1}T_1(x)=\frac{x^2-1}{2}, \\ && \p_x^{-1}T_n(x)=\frac{T_{n+1}(x)}{2(n+1)}-\frac{T_{n-1}(x)}{2(n-1)}-\frac{(-1)^n}{n^2-1},\quad n\geq 2.\nn\eeqa
and
\beqa && \p_x^{-2}T_0(x)=\frac{(x+1)^2}{2},\nn \\ && \p_x^{-2}T_1(x)=\frac{(x-2)(x+1)^2}{6},\nn \\ && \p_x^{-2}T_2(x)=\frac{x(x-2)(x+1)^2}{6}\label{px2} \\ && \p_x^{-2}T_n(x)=\frac{T_{n+2}(x)}{4(n+1)(n+2)}-\frac{T_n(x)}{2(n^2-1)}+\frac{T_{n-2}(x)}{4(n-1)(n-2)}\nn \\&& \hspace{20mm} -\frac{(-1)^n(1+x)}{n^2-1}-\frac{3(-1)^n}{(n^2-1)(n^2-4)},\quad n\geq 3.\nn\eeqa

\section{Pseudospectral integration matrices} 
Given $\{y_j\}_{j=0}^{M}$ and $\{c_j\}_{j=0}^{M+m}$ with $m\geq 1$, we consider the Birkhoff-type  interpolation problem: \beqs\label{birkhoff} {\rm Find}\ p(x)\in\mbbp_{M+m}\ {\rm such \ that}\ \l\{ \begin{array}{ll} p^{(m)}(y_j)=c_j,& j=0,\ldots,M,  \\ \mcall_i\l(p,\ldots,p^{(m-1)}\r)= c_{M+i}, &   i=1,\ldots,m,\end{array}\r.\eeqs where each $\mcall_i$ is a linear functional. Let $\{\ell_{j,M}(x)\}_{j=0}^{M}$ be the Lagrange interpolation basis polynomials of degree $M$ associated with the points $\{y_j\}_{j=0}^{M}$. Then the Birkhoff-type interpolation polynomial takes the form \beq\label{form} p(x)=\sum_{j=0}^Mc_j\p_x^{-m}\ell_{j,M}(x)+\sum_{i=0}^{m-1}{\alpha_i}x^i,\eeq where $\alpha_i$ can be determined by the linear constraints $\mcall_i(p,\ldots,p^{(m-1)})= c_{M+i}.$ Obviously, the existence and uniqueness of the Birkhoff-type interpolation polynomial is equivalent to that of $\{\alpha_i\}_{i=0}^{m-1}$. After obtaining $\alpha_i$, we can rewrite (\ref{form}) as \beqs\label{bip}p(x)=\sum_{j=0}^{M+m}c_jB_j(x),\qquad B_j(x)\in\mbbp_{M+m}.\eeqs 

Let $N=M+m$ and $\{x_i\}_{i=0}^N$ be the points as in (\ref{iptsx}). Define the $m$th-order {\it pseudospectral integration matrix} (PSIM) as:
\beqs\label{psim}{\bf B}^{(-m)}_{{\bf y}\mapsto {\bf x}}=\l[B_j(x_i)\r]_{i,j=0}^N.\eeqs Define the matrices $${\bf B}^{(k-m)}_{{\bf y}\mapsto {\bf x}}=\l[B_{j}^{(k)}(x_i)\r]_{i,j=0}^N,\qquad k\geq 1.$$ It is easy to show that \beq\label{dkqm}{\bf B}^{(k-m)}_{{\bf y}\mapsto {\bf x}}={\bf D}^{(k)}_{{\bf x}\mapsto {\bf x}}{\bf B}^{(-m)}_{{\bf y}\mapsto {\bf x}},\qquad k\geq 1.\eeq

Let ${\bf L}_m$ be the discretization of the linear constraints $\mcall_i$, $1\leq i\leq m$. We have the following theorem.
 
\begin{theorem}\label{main} If for any $p(x)\in\mbbp_N$, $$\l[\begin{array}{c}\mcall_1\l(p,\ldots,p^{(m-1)}\r)\\ \vdots \\ \mcall_m\l(p,\ldots,p^{(m-1)}\r) \end{array}\r]= {\bf L}_m\l[\begin{array}{c} p(x_0) \\ \vdots \\ p(x_N)\end{array}\r],$$ then \beqs\label{identity} \l[\begin{array}{c}{\bf D}^{(m)}_{{\bf x}\mapsto {\bf y}}\\ {\bf L}_m\end{array}\r]{\bf B}^{(-m)}_{{\bf y}\mapsto {\bf x}}={\bf I}_{N+1}. \eeqs \end{theorem} \begin{proof} The result follows from \beqs{\bf D}^{(m)}_{{\bf x}\mapsto {\bf x}}{\bf B}^{(-m)}_{{\bf y}\mapsto {\bf x}}=\l[\begin{array}{ll} {\bf P}^{{\bf y}\mapsto {\bf x}} & {\bm 0}\end{array}\r],\qquad {\bf L}_m{\bf B}^{(-m)}_{{\bf y}\mapsto {\bf x}}=\l[\begin{array}{ll}{\bm 0}& {\bf I}_m\end{array}\r],\eeqs and Lemma \ref{mnnm}. \end{proof}
 
Now we give concrete examples. Consider the non-separable linear constraint 
 \beq\label{lc1}ap(-1)+bp(1)=\sigma,\eeq and the global linear constraint \beq\label{lc2} \int_{-1}^1p(x)\rmd x=\sigma,\eeq
 where $a$, $b$ and $\sigma$ are given constants. They are straightforward to discretize: for (\ref{lc1}), $$\l[\begin{array}{lllll}a& 0& \cdots& 0& b\end{array} \r]{\bf p}=\sigma$$ and for (\ref{lc2}), $${\bf q}^\rmt{\bf p}={\bf\sigma},$$ where $${\bf p}=\l[\begin{array}{cccc}p(x_0) & p(x_1) & \cdots & p(x_N)\end{array}\r]^\rmt$$ and ${\bf q}$ is a column vector of Clenshaw-Curtis quadrature weights \cite{funaro1992polyn}.
  
The first-order Birkhoff-type interpolation problem takes the form: \beqs\label{birkhoff1}{\rm Find} \ \ p(x)\in\mbbp_{M+1} \ \ {\rm such \ that} \ \ \l\{ \begin{array}{ll} p'(y_j)=c_j,& j=0,1,\ldots,M,  \\ \mcall p= c_{M+1}. & \end{array}\r.\eeqs  
\begin{itemize}


\item Given $\mcall p:= ap(-1)+bp(1)$  with $a+b\neq 0$, we have
\beqas && B_j(x)=\p_x^{-1}\ell_{j,M}(x)-\frac{b}{a+b}\int_{-1}^1\ell_{j,M}(x)\rmd x, \quad j=0,1,\ldots,M,\\ &&  B_{M+1}(x)=\frac{1}{a+b}.\eeqas 

\item Given $\dsp\mcall p:= \int_{-1}^1 p(x)\rmd x$, we have \beqas && B_j(x)=\p_x^{-1}\ell_{j,M}(x)-\frac{1}{2}\int_{-1}^1\p_x^{-1}\ell_{j,M}(x)\rmd x, \quad j=0,1,\ldots,M,\\ &&  B_{M+1}(x)=\frac{1}{2}.\eeqas 
\end{itemize}
By (\ref{ltp}) and (\ref{px1}), the matrix ${\bf B}^{(-1)}_{{\bf y}\mapsto {\bf x}}$ can be computed stably even for thousands of collocation points.

The second-order Birkhoff-type interpolation problem takes the form: \beqs\label{birkhoff2}{\rm Find} \ \ p(x)\in\mbbp_{M+2} \ \ {\rm such \ that} \ \ \l\{ \begin{array}{ll} p''(y_j)=c_j,& j=0,1,\ldots,M,  \\ \mcall_i\l(p,p'\r)=c_{M+i}, & i=1,2.\end{array}\r.\eeqs    
\begin{itemize}

\item Given $\mcall_1(p,p'):=ap(-1)+bp(1)$ with $a\neq b$, and $\mcall_2(p,p')=\dsp\int_{-1}^1p(x)\rmd x$, we have 
\beqas && B_j(x)=\p_x^{-2}\ell_{j,M}(x)-\frac{bx}{b-a}\int_{-1}^1\p_x^{-1}\ell_{j,M}(x)\rmd x
\\ &&\qquad\qquad +\l(\frac{(a+b)x}{2(b-a)}-\frac{1}{2}\r)\int_{-1}^1\p_x^{-2}\ell_{j,M}(x)\rmd x,\quad j=0,1,\ldots,M, \\ && B_{M+1}(x)=\frac{x}{b-a}, \\ && B_{M+2}(x)=\frac{1}{2}-\frac{(a+b)x}{2(b-a)},
\eeqas and \beqas && B_j'(x)=\p_x^{-1}\ell_{j,M}(x)-\frac{b}{b-a}\int_{-1}^1\p_x^{-1}\ell_{j,M}(x)\rmd x
\\ &&\qquad\qquad +\frac{a+b}{2(b-a)}\int_{-1}^1\p_x^{-2}\ell_{j,M}(x)\rmd x,\quad j=0,1,\ldots,M, \\ && B_{M+1}'(x)=\frac{1}{b-a}, \\ && B_{M+2}'(x)=-\frac{a+b}{2(b-a)}.
\eeqas
\end{itemize} 
By (\ref{ltp}), (\ref{px1}) and (\ref{px2}), the matrices ${\bf B}^{(-2)}_{{\bf y}\mapsto {\bf x}}$ and ${\bf B}^{(1-2)}_{{\bf y}\mapsto {\bf x}}$ can be computed stably even for thousands of collocation points.

\section{Preconditioning rectangular spectral collocation}
Consider the $m$th-order differential equations of the form \beq\label{mth} a_{m}(x)u^{(m)}(x)+\cdots+a_1(x)u'(x)+a_0(x)u(x)=f(x),
\eeq together with linear constraints \beq\label{mc} \mcall_i(u,\ldots,u^{(m-1)})=c_{M+i},\qquad i=1,2,\ldots,m.\eeq Let $\{x_j\}_{j=0}^N$ (with $N=M+m$) and $\{y_j\}_{j=0}^M$ be the points as  defined in (\ref{iptsx}) and (\ref{iptsy}), respectively.
The rectangular spectral collocation discretization \cite{driscoll2015recta} of (\ref{mth}) is given by $${\bf A}_{M+1}{\bf u}={\bf f},$$ where $${\bf A}_{M+1}=\diag\{{\bf a}_m\}{\bf D}^{(m)}_{{\bf x}\mapsto {\bf y}}+\cdots+
\diag\{{\bf a}_1\}{\bf D}^{(1)}_{{\bf x}\mapsto {\bf y}}+\diag\{{\bf a}_0\}{\bf D}^{(0)}_{{\bf x}\mapsto {\bf y}}.$$ Here we use boldface letters to indicate a column vector obtained by discretizing at the points $\{y_j\}_{j=0}^M$ except for the unknown ${\bf u}$. For example, \beqas &&{\bf a}_0=\l[\begin{array}{cccc}a_0(y_0)& a_0(y_1) & \cdots & a_0(y_M)\end{array}\r]^\rmt, \\ &&{\bf f}=\l[\begin{array}{cccc}f(y_0)& f(y_1) & \cdots & f(y_M)\end{array}\r]^\rmt.\eeqas Let $${\bf L}_m{\bf u}={\bf c}_m$$ be the discretization of the linear constraints (\ref{mc}) and satisfy the condition in Theorem \ref{main}, where $${\bf c}_m=\l[\begin{array}{cccc}c_{M+1} & c_{M+2} & \cdots & c_{M+m}\end{array}\r]^\rmt.$$ The global collocation system is given by \beq\label{gcs} {\bf Au=g}\eeq where \beqas &&{\bf A}=\l[\begin{array}{c}{\bf A}_{M+1}\\ {\bf L}_m\end{array}\r],\quad  {\bf g}=\l[\begin{array}{c} {\bf f}\\ {\bf c}_m\end{array}\r].\eeqas

Consider the pseudospectral integration matrix ${\bf B}^{(-m)}_{{\bf y}\mapsto {\bf x}}$ (\ref{psim}) as a right preconditioner for the linear system (\ref{gcs}). We need to solve the right preconditioned linear system \beqs{\bf A}{\bf B}^{(-m)}_{{\bf y}\mapsto {\bf x}}{\bf v}={\bf g}.\eeqs 
By (see Theorem \ref{main}) $${\bf L}_m{\bf B}^{(-m)}_{{\bf y}\mapsto {\bf x}}=\l[\begin{array}{ll}{\bm 0}& {\bf I}_m\end{array}\r],$$ we have 
\beq\label{modal2}{\bf A}_{M+1}{\bf B}^{(-m)}_{{\bf y}\mapsto {\bf x}}\l[\begin{array}{c}{\bf I}_{M+1} \\ {\bm 0}\end{array}\r]{\bf v}_{M+1}={\bf f}-{\bf A}_{M+1}{\bf B}^{(-m)}_{{\bf y}\mapsto {\bf x}}\l[\begin{array}{c}{\bm 0}\\ {\bf I}_{m} \end{array}\r]{\bf c}_m.\eeq
There hold \beqas&& {\bf A}_{M+1}{\bf B}^{(-m)}_{{\bf y}\mapsto {\bf x}}\l[\begin{array}{c}{\bf I}_{M+1} \\ {\bm 0}\end{array}\r]\\&=&\diag\{{\bf a}_m\}+\diag\{{\bf a}_{m-1}\}\wt{\bf B}^{(m-1-m)}_{{\bf y}\mapsto {\bf y}}+\cdots+\diag\{{\bf a}_0\}\wt{\bf B}^{(0-m)}_{{\bf y}\mapsto {\bf y}},\eeqas and \beqs
{\bf A}_{M+1}{\bf B}^{(-m)}_{{\bf y}\mapsto {\bf x}}\l[\begin{array}{c}{\bm 0}\\ {\bf I}_{m} \end{array}\r]= \diag\{{\bf a}_{m-1}\}\wh{\bf B}^{(m-1-m)}_{{\bf y}\mapsto {\bf y}}+\cdots+\diag\{{\bf a}_0\}\wh{\bf B}^{(0-m)}_{{\bf y}\mapsto {\bf y}},\eeqs where, for $k=0,1,\cdots,m-1$, $$\wt{\bf B}^{(k-m)}_{{\bf y}\mapsto {\bf y}}=\l[B_j^{(k)}(y_i)\r]_{i,j=0}^M,\qquad \wh{\bf B}^{(k-m)}_{{\bf y}\mapsto {\bf y}}=\l[B_j^{(k)}(y_i)\r]_{i=0,j=M+1}^{M,M+m}.$$
After solving (\ref{modal2}), we obtain ${\bf u}$ by $${\bf u}={\bf B}^{(-m)}_{{\bf y}\mapsto {\bf x}}\l[\begin{array}{c} {\bf v}_{M+1} \\ {\bf c}_m \end{array}\r].$$

\section{Numerical results}

In this section, we compare the rectangular spectral collocation (RSC) scheme (\ref{gcs}) and the preconditioned rectangular spectral collocation (P-RSC) scheme (\ref{modal2}). 
In all computations, the Chebyshev points of the second kind are chosen as $\{x_j\}_{j=0}^N$ and the Chebyshev points of the first kind are chosen as $\{y_j\}_{j=0}^M$. 

\subsection*{Example 1} 
We consider the equation \beq\label{example1} u'(x)+a_0(x)u(x)=f(x)\eeq with the linear constraint
 \beq\label{ulc1}u(-1)+u(1)=\sigma,\eeq or the linear constraint \beq\label{ulc2} \int_{-1}^1u(x)\rmd x=\sigma,\eeq
 where $\sigma$ is a given constant. We report in Table \ref{e1t1} the condition numbers of the linear systems in RSC and P-RSC with $a_0(x)=2x, -\sin x$, and various $N$. We observe that the condition numbers of P-RSC are independent of $N$, while those of RSC behave like $\mcalo(N^{2.5})$.

\begin{table}[htp]
\caption{Comparison of condition numbers.}
\label{e1t1}
\begin{center} \footnotesize
\begin{tabular}{c|c|c|c|c|c|c|c|c} \toprule
&\multicolumn{4}{c|}{$a_0(x)=2x$}& \multicolumn{4}{c}{$a_0(x)=-\sin x$}\\ \hline
&\multicolumn{2}{c|}{Constraint (\ref{ulc1})}& \multicolumn{2}{c|}{Constraint (\ref{ulc2})}&\multicolumn{2}{c|}{Constraint (\ref{ulc1})}& \multicolumn{2}{c}{Constraint (\ref{ulc2})}\\ \hline
 $N$& RSC& P-RSC& RSC & P-RSC& RSC& P-RSC& RSC & P-RSC\\ \hline
 128& 6.86e+04 & 3.19 &  3.37e+04 & 2.54 & 3.04e+04 & 1.95 & 4.07e+04 & 1.95\\
 256& 3.87e+05 & 3.19 &  1.91e+05 & 2.54 & 1.72e+05 & 1.95 & 2.29e+05 & 1.95\\
 512& 2.19e+06 & 3.19 &  1.08e+06 & 2.54 & 9.68e+05 & 1.95 & 1.30e+06 & 1.95\\
1024& 1.24e+07 & 3.19 &  6.10e+06 & 2.54 & 5.47e+06 & 1.95 & 7.32e+06 & 1.95\\ \bottomrule
\end{tabular}
\end{center}
\end{table}

\begin{figure}[!htpb]
\centerline{\epsfig{figure=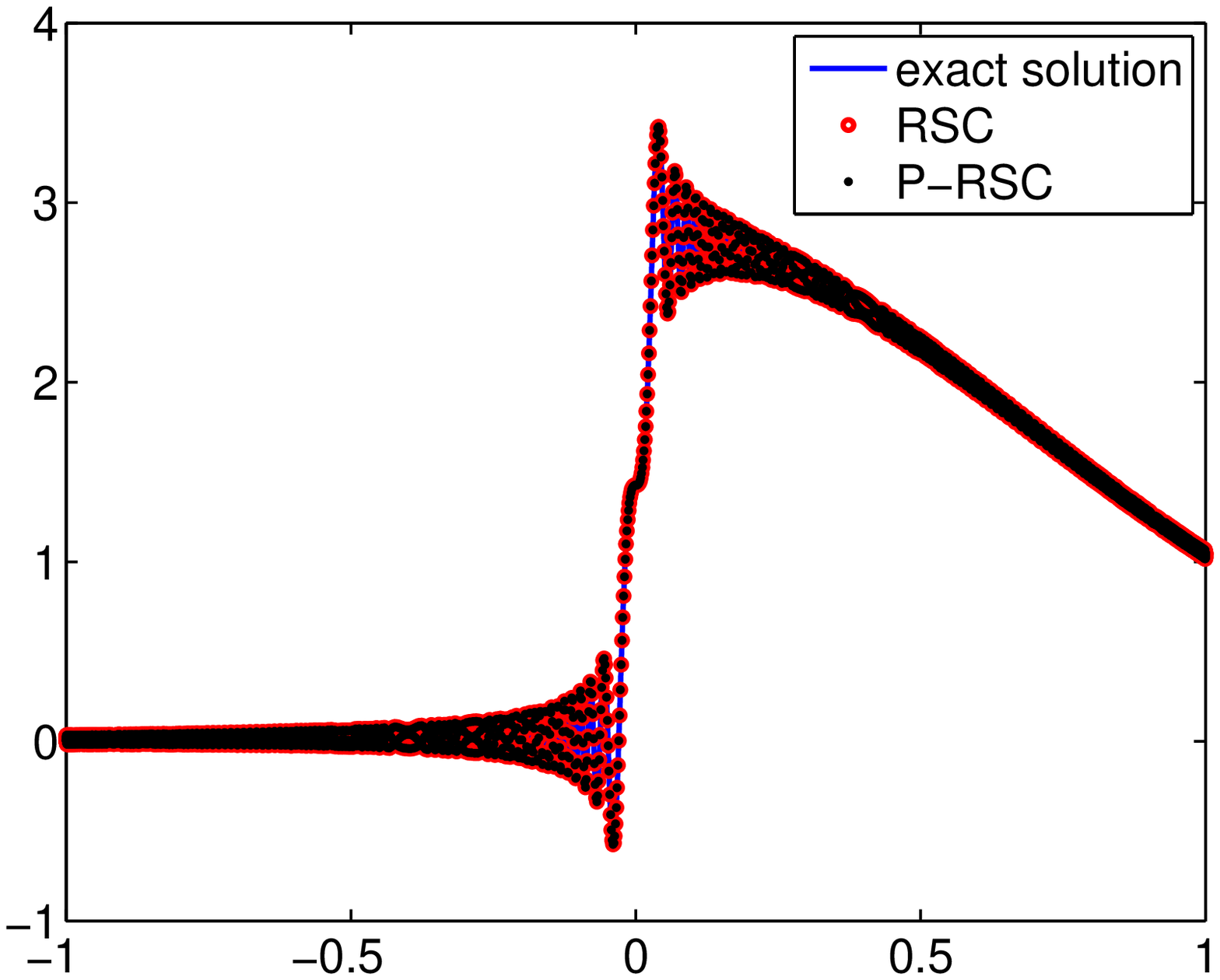,height=2.0in}\epsfig{figure=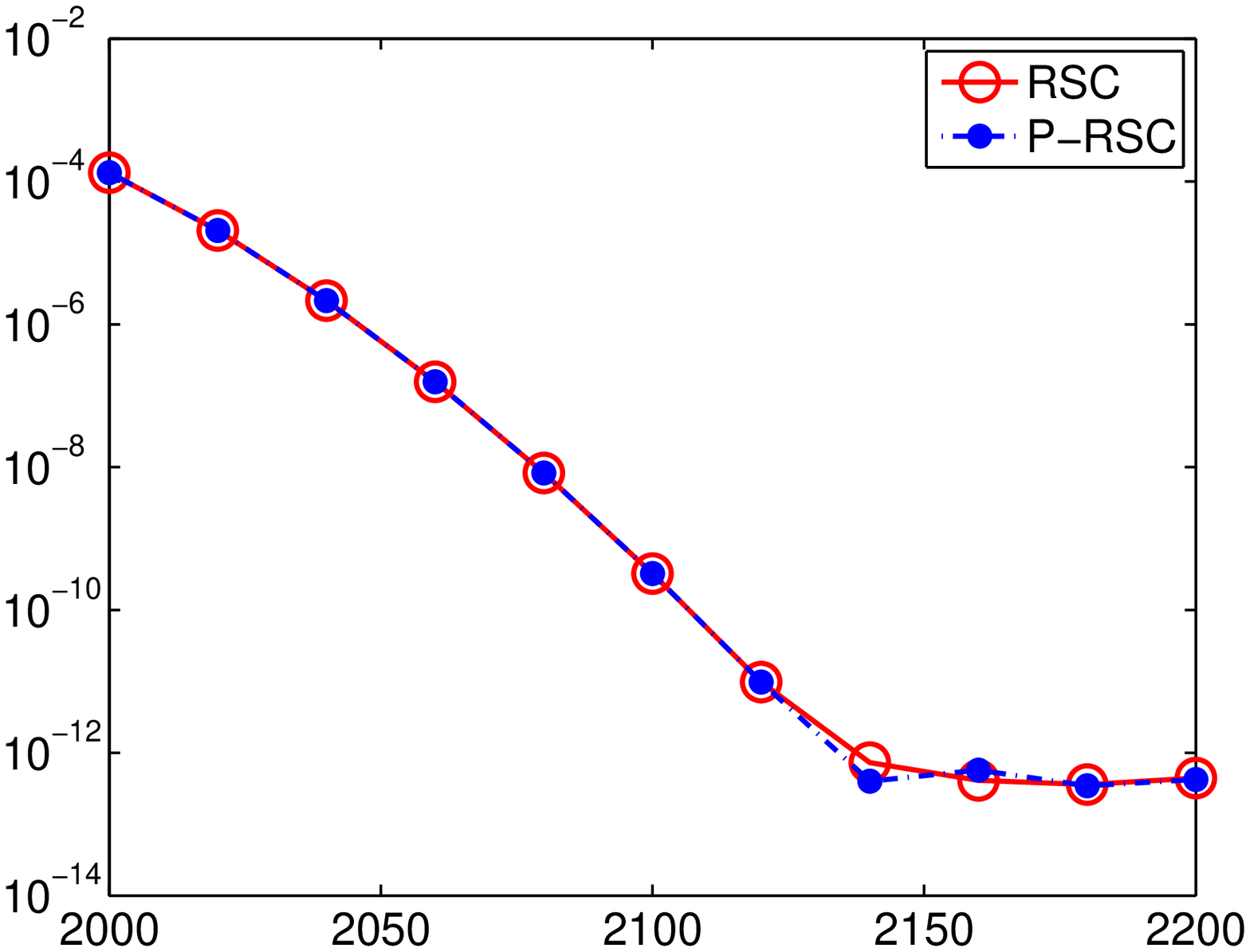,height=2.0in}}
\centerline{\footnotesize\hspace{3mm}(a)\hspace{6.3cm} (b)}
\caption{{\rm (a):} exact solution versus numerical solutions. {\rm (b):} comparison of numerical errors}\label{e1f1}
\end{figure} 

We next consider (\ref{example1}) with $a_0(x)=2x$ and the linear constraint (\ref{ulc1}). The function $f(x)$ and $\sigma$ are chosen such that an oscillatory solution  of (\ref {example1}) is $$ u(x)=100\exp(-x^2)\int_{-1}^x\exp(t^2)\sin(2000t^2)\rmd t.$$ In Figure \ref{e1f1} (a) we plot the exact solution against the numerical solutions obtained by RSC and P-RSC with $N=2200$. In Figure \ref{e1f1} (b) we plot the maximum point-wise errors of RSC and P-RSC. It indicates that for this example, even for very large $N$, both RSC and P-RSC are very stable.

\subsection*{Example 2} 
We consider the equation \beq\label{example2} \ve u''(x)-xu'(x)-u(x)=f(x)\eeq with the linear constraints  \beqs u(-1)-u(1)=\sigma_1,\qquad \int_{-1}^1u(x)\rmd x=\sigma_2.\eeqs The function $f(x)$, $\sigma_1$ and $\sigma_2$ are chosen such that the exact solution of (\ref {example2}) is $$ u(x)=\exp\l(\frac{x^2-1}{2\ve}\r).$$ 
 
\begin{table}[htp]
\caption{Comparison of condition numbers and iterations of different schemes for $\ve=1$.}
\label{e2t1}
\begin{center} \footnotesize
\begin{tabular}{c|c|c|c|c|c|c} \toprule
&\multicolumn{3}{c|}{RSC}& \multicolumn{3}{c}{P-RSC}\\ \hline
 $N$& Condition& Error& Iterations & Condition& Error & Iterations\\ \hline
 128& 1.95e+08 & 8.41e-10 &$>$1000 & 2.73 & 6.66e-16 & 8\\
 256& 4.39e+09 & 9.23e-09 &$>$1000 & 2.73 & 6.66e-16 & 8\\
 512& 9.94e+10 & 7.84e-08 &$>$1000 & 2.73 & 8.88e-16 & 8\\
1024& 2.25e+12 & 2.49e-06 &$>$1000 & 2.73 & 1.11e-15 & 8\\ \bottomrule
\end{tabular}
\end{center}
\end{table} 
 
In Tables \ref{e2t1}-\ref{e2t3}, we present the condition numbers, the maximum point-wise errors, and the number of iterations via the GMRES algorithm \cite{saad1986gmres} with the relative tolerance equal to $10^{-10}$ and the restart number equal to $40$, for the cases $\ve=1$, $\ve=0.1$, and $\ve=0.01$, respectively. 
We observe that the condition numbers of P-RSC are independent of $N$, while those of RSC behave like $\mcalo(N^{4.5})$.

\begin{table}[htp]
\caption{Comparison of condition numbers and iterations of different schemes for $\ve=0.1$.}
\label{e2t2}
\begin{center} \footnotesize
\begin{tabular}{c|c|c|c|c|c|c} \toprule
&\multicolumn{3}{c|}{RSC}& \multicolumn{3}{c}{P-RSC}\\ \hline
 $N$& Condition& Error& Iterations & Condition& Error & Iterations\\ \hline
 128& 6.74e+07 & 2.65e-10 &$>$1000 & 5.11e+02 & 1.14e-14 & 16\\
 256& 1.50e+09 & 5.95e-10 &$>$1000 & 5.11e+02 & 1.62e-14 & 16\\
 512& 3.35e+10 & 4.12e-09 &$>$1000 & 5.11e+02 & 1.58e-14 & 16\\
1024& 7.55e+11 & 1.69e-07 &$>$1000 & 5.11e+02 & 1.49e-14 & 16\\ \bottomrule
\end{tabular}
\end{center}
\end{table}

\begin{table}[htp]
\caption{Comparison of condition numbers and iterations of different schemes for $\ve=0.01$.}
\label{e2t3}
\begin{center} \footnotesize
\begin{tabular}{c|c|c|c|c|c|c} \toprule
&\multicolumn{3}{c|}{RSC}& \multicolumn{3}{c}{P-RSC}\\ \hline
 $N$& Condition& Error& Iterations & Condition& Error & Iterations\\ \hline
 128& 4.47e+07 & 2.23e-10 &$>$1000 & 3.70e+05 & 3.11e-13 & 64\\
 256& 9.77e+08 & 2.20e-09 &$>$1000 & 3.70e+05 & 1.04e-12 & 65\\
 512& 2.16e+10 & 7.95e-09 &$>$1000 & 3.70e+05 & 1.34e-12 & 67\\
1024& 4.84e+11 & 4.69e-07 &$>$1000 & 3.70e+05 & 5.35e-13 & 67\\ \bottomrule
\end{tabular}
\end{center}
\end{table}

\section{Concluding remarks} We have proposed a preconditioning rectangular spectral collocation scheme for $m$th-order ordinary differential equations together with $m$ general linear constraints. The condition number of the resulting linear system is typically independent of the number of collocation points. And the linear system can be solved by an iterative solver within a few iterations. The application of the preconditioning scheme to nonlinear problems is straightforward. 

%

\section*{Acknowledgment} We thank Prof. Li-Lian Wang (Nanyang
Technological University, Singapore) for providing the MATLAB codes used in \cite{wang2014well}.


\end{document}